\documentclass[12pt]{amsart}%
\usepackage{amsmath}
\usepackage{amssymb}
\usepackage{amsthm}
\usepackage{amscd}
\usepackage[dvipdfmx]{graphicx}
\usepackage[mathscr]{eucal}
\usepackage{bm}
\usepackage{xcolor}
\makeindex
\newtheorem{Thm}{Theorem}[section]
\theoremstyle{plain}
\newtheorem{Theorem}[Thm]{Theorem}
\newtheorem{Lemma}[Thm]{Lemma}
\newtheorem{Corollary}[Thm]{Corollary}
\newtheorem{Proposition}[Thm]{Proposition}
\theoremstyle{definition}
\newtheorem{Definition}[Thm]{Definition}
\newtheorem{Example}[Thm]{Example}
\theoremstyle{remark}
\newtheorem{Remark}{Remark}

\font\ym=msbm10

\newcommand{\Z}{\text{\ym Z}}

\newcommand{\R}{\text{\ym R}}

\newcommand{\C}{\text{\ym C}}

\newcommand{\sfrac}[2]{\mbox{\small $\dfrac{#1}{#2}$}}
\begin{document}
\title[Shifting Semicircular Roots]{On Shifting Semicircular Roots}
\author{Shigeru Yamagami}
\address{Graduate School of Mathematics, Nagoya University}
\urladdr{https://www.math.nagoya-u.ac.jp/~yamagami/}
\author{Hiroaki Yoshida}
\address{Department of Information Sciences, Ochanomizu University}
\begin{abstract}
  In the framework of continued fraction expansions of Stieltjes transforms, we consider shifting of semicircular laws.
  The continuous part of the associated measure admits a density function which is the quotient of semicircular one by a polynomial.
  We study how this polynomial denominator determines shifted semicircular laws,
  with explicit descriptions in examples of shifting up to the level of step two. 
\end{abstract}

\maketitle


\section*{Introduction}
In recent developments in free probability theory, various operations on probability distributions are effectively utilized to
organize relevant probability laws in a perspective way.
Among them the most basic one is Wigner's semicircular law and its deformed versions are found a lot
in connection with suitable models.
One series of examples belonging to this category is obtained by the first author during studies of the spectral properties of
Haagerup's positive definite functions on free groups, which turns out to be a distinctive form of modified semicircular law.

We here continue considerations of Haagerup type distributions from the view point of the associated continued fraction expansion.
Based on somewhat lengthy but routine computations, the Cauchy-Stieltjes transforms of the distrubutions in question are shown 
to have continued fraction expansions which are just deformations of the semicircular one up to the level two
(or equivalently it is a two-step shift of semicircular roots). 
Similar phenomena are observed in other interesting defomrmations, which made us to think seriously about its mechanism.

More explicitly, we assume that our Stieltjs transforms take the form $(F(w) + \sqrt{w^2 - 4c})/G(w)$ with $F$ and $G$ polynomials of
a complex variable $w$. In the case of Haagerup type distributions, $F$ and $G$ are polynomials of degree $2$ and $3$ respectively.
If these really come from probability distributions, the Stiltjes inversion formula reveals that, except for finitely many point masses,
the continuum part of the distribution is given by a density function proportional to $\sqrt{4c-t^2}/G(t)$ ($-2\sqrt{c} < t < 2\sqrt{c}$).

Here arises our basic question:
To what extent does the density function determine measures whose Stieltjes transforms take the form of repeated shifts of semicircular roots.

We first observe that, though on a formal algebraic level,
the roots of $G$ determines the polynomial $F$ as well as the proportional constant of $G$. 

The observation is then explicitly worked out in one-step shifting of semicircular roots.
Two-step shifting is also checked when $G$ is divided by $w^2 - \zeta^2$ with $\zeta$ a complex constant.

The positivity of Stieltjes transforms in these classes are also described, inclusing Haagerup type distributions
as a special case.

In the last section, relations to several known operations on free probability distributions are discussed. 

The authors are grateful to M.~Nagisa and M.~Uchiyama for fruitful discussions on the subject on various occasions.

\section{Generalities}
Given a complex measure $\mu$ in $\R$, its Stieltjes transform $S_\mu(w)$ is a holomorphic function of $w \in \C \setminus \R$ defined by
\[
  S_\mu(w) = \int_\R \frac{1}{t-w}\, \mu(dt), 
\]
which contains full information of $\mu$ because it restores $\mu$ by the inversion formula
\[
  2\pi i \mu(dt) = \lim_{\epsilon \to +0} \Bigl(S_\mu(t+i\epsilon) - S_\mu(t-i\epsilon)\Bigr)\, dt. 
\]
Here convergence on the right hand side is in the weak* sense in the dual Banach space $C_0(\R)^*$
{($C_0(\R)$ being the Banach space of
complex-valued continuous functions on $\R$ vanishing at infinity with the uniform norm)}.
Since $\overline{S_\mu(w)} = S_{\overline{\mu}}(\overline{w})$, $\mu$ is real (or a signed measure)
if and only if $S = S_\mu$ is \textbf{real} in the sense that $\overline{S(w)} = S(\overline{w})$.

The Stieltjes transform $S$ of a probability measure $\mu$ 
is then characterized as a real holomorphic function $S$ on $\C \setminus \R$ satisfying 
\begin{enumerate}
  \item (positivity)
    $\text{Im}(S(w)) > 0$ ($\text{Im}\, w > 0$) and
  \item (normalization)
    $\displaystyle \lim_{y \to +\infty} y\, S(iy) = i$.
  \end{enumerate}

When $\mu$ is a probability measure for which polynomial functions are integrable, its Stieltjes transform
is known to be expressed by a continued fraction
 \[
\cfrac{1}
{a_0-w + \cfrac{-b_0^2}
{a_1-w + \cfrac{-b_1^2}
{a_2-w + \ddots}
}}
= \cfrac{-1}
{w-a_0 - \cfrac{b_0^2}
{w-a_1- \cfrac{b_1^2}
{w-a_2- \ddots}
}},  
\]
where a sequence $a = (a_n)_{n \geq 0}$ of reals and
a sequence $b = (b_n)_{n \geq 0}$ of strictly positive reals (referred to as Jacobi parameters) are
coefficients of recurrence relation among orthogonal polynomials associated to
$\mu$ and constitute a diagonal and an off-diagonal parts of the so-called Jacobi matrix $J_{a,b}$, 
which is real-symmetric and calculated via moment sequences in such a way that they are in one-to-one correspondence.
  
It is also known that the following conditions are equivalent: 
\begin{enumerate}
\item
  $J_{a,b}$ is bounded as an operator on $\ell^2$.
\item
  $a= (a_n)$ and $b = (b_n)$ are bounded sequences.
\item
  $\mu$ is supported by a bounded subset of $\R$.
\item
  The Stieltjes transform $S_\mu(w)$ is analytic at $w = \infty$. 
\end{enumerate}

Moreover, if $E(dt)$ is the (projection-valued) spectral measure of $J_{a,b}$ on $\R$, then $\mu(dt) = (e_0|E(dt)e_0)$ and hence 
\[
  S_\mu(w) = \Bigl(e_0\Bigl| \frac{1}{J_{a,b} - wI}\Bigr.e_0\Bigr). 
\]
Here $e_0 = (1,0,\cdots)$ is a unit vector in $\ell^2$ and $I$ denotes the identity operator. 

Notice that, if this is the case, the Stieltjes transform $S(w)$ which is analytic at $w=\infty$ inductively determines
coefficients $a_0,a_1,\dots$ and $b_0,b_1,\dots$ as follows:

From the asymptotic behavior of $S(w)$ at $w=\infty$, we have
\[
  \frac{1}{S(w)} = - w + a_0 + O(1/w)
\]
and then, in view of $w - a_0 + 1/S(w) = O(1/w)$,  
\[
  \frac{1}{w - a_0 + 1/S(w)} = \alpha + \beta w + O(1/w).
\]
If $\beta \not= 0$, letting $b_0 = 1/\sqrt{\beta}$ and $a_1 = - \alpha/\beta$, an analytic function
\[
  S_1(w) = -\beta(w-a_0 + \frac{1}{S}) = -\beta\frac{(w-a_0)S + 1}{S}
\]
behaves like $-1/w + O(1/w^2)$ at $w = \infty$ and we can repeat the procedure inductively to get $a_n$, $b_n$ and $S_n(w)$ so that
\[
  \frac{1}{S_{n-1}(w)} + w - a_{n-1} = -b_{n-1}^2 S_n(w),
  \quad
  \frac{1}{S_n(w)} = a_n - w + O(1/w) 
  \]
  as long as $b_n \not= 0$.

  \begin{Example}
    Let
    \[
      S(w) = \frac{1}{b} \frac{b-w + \sqrt{w^2-4c}}{a-w}.
    \]
    Then
    \[
      \frac{1}{S} = - w + \frac{ab - 2c}{b} + O(1/w)
    \]
    and
    \[
      \frac{1}{w - (ab-2c)/b + 1/S} = \frac{b^2}{2c(ab-2c)} w + \frac{b(b^2-2ab + 4c)}{2(ab-2c)^2} + O(1/w), 
    \]
    whence
    \[
      b_0 = \frac{\sqrt{2c(ab-2c)}}{b},
        \quad
        a_1 = - \frac{(b^2 - 2ab + 4c)c}{(ab - 2c)b}. 
      \]
  \end{Example}

  In terms of sequences $(a_n)$ and $(b_n)$ obtained in this way from $S(w)$,
  the reality condition $\overline{S(w)} = S(\overline{w})$ is equivalent to $a_n, b_n^2 \in \R$ ($n \geq 0$).
  A real $S$ is then said to be \textbf{positive} if $b_n^2 > 0$ ($n \geq 0$).
  Given a positive $S$ with $(a_n)$ and $(b_n)$ bounded, we have a finite measure $\mu$ such that its support $[\mu]$ is bounded and
  Jacobi parameters are given by $(a_n)$ and $(b_n)$.
  The Stieltjes transform $S_\mu(w)$ of $\mu$ is then given by $S(w)$ near $w=\infty$ and hence $S_\mu = S$ globally.
  
  In particular $S$ is holomorphic on $\overline{\C} \setminus [\mu]$ and the measure $\mu$ of bounded support is restored from $S$
  by the Stieltjes inversion formula. 

\bigskip
Given a probability measure $\mu$ in $\R$, its affine transform $\mu_{a,b}$ ($a \in \R$, $0 \not=b \in \R$) defined by
\[
  \int_\R f(t)\, \mu_{a,b}(dt) = \int_\R f(bt + a)\, \mu(dt)
\]
satisfies
\[
  S_{\mu_{a,b}}(w) = \int_\R \frac{1}{t-w}\, \mu_{a,b}(dt)
  = \int_\R \frac{1}{bt+a - w}\, \mu(dt)
  = \frac{1}{b} S_\mu((w-a)/b)
\]
with its continued fraction given by
 \[
\cfrac{1}
{ba_0 + a -w + \cfrac{-b^2b_0^2}
{ba_1 + a -w + \cfrac{-b^2b_1^2}
{ba_2 + a -w + \ddots}
}}. 
\]
Thus $(a_n)$ and $(b_n)$ are changed to $(ba_n +a)$ and  $(|b|b_n)$. 

\begin{Example}
$\mu$ is symmetric, i.e., invariant under $t \leftrightarrow -t$, if and only if $a_j = 0$ ($j \geq 0$). 
\end{Example}


\section{Semicircular Roots}
Given a positive real $c > 0$ and a complex parameter $w$,
consider a quadratic equation $z^2 + wz + c = 0$ of $z$ with its solutions denoted by
  \begin{align*}
    \rho &= \frac{-w + \sqrt{w^2 - 4c}}{2} = -\frac{c}{w} + O(1/w^2),\\
    \rho^* &= \frac{-w - \sqrt{w^2 - 4c}}{2} = -w + \frac{c}{w} + O(1/w^2). 
  \end{align*}
  Here the square root $\sqrt{z^2 - 4c}$ of $z^2 - 4c$ is an analytic function of $z \in \overline{\C} \setminus [-2\sqrt{c},2\sqrt{c}]$ 
  with its root branch specified by  $\sqrt{z^2-4c} = z(1-2c/z^2) + O(1/z^2)$ for a large $z \in \C$ and the boundary value on $\R$ given by 
  \[
    \sqrt{(x\pm i0)^2 - 4c}
    =
    \begin{cases}
      \sqrt{x^2 - 4c} &(x> 2\sqrt{c})\\
      \pm i\sqrt{4c - x^2} &(-2\sqrt{c} < x < 2\sqrt{c})\\
      - \sqrt{x^2-4c} &(x < - 2\sqrt{c})
      \end{cases}. 
    \]

    \begin{Remark}
      $-\rho$ and $-\rho^*$ for $c=1$ are actually inverse maps of the Joukowski transform in aerodynamics.
\end{Remark}

    The quadratic root $\rho$ (called \textbf{semicircular root}) then maps the upper (lower) half plane onto an upper (lower) semidisk
    so that
    the line segment $x\pm i0$ ($-2\sqrt{c} \leq x \leq 2\sqrt{c})$ gives a parametric 
    upper (lower) semicircle $(-x \pm i\sqrt{4c-x^2})/2$ of radius $\sqrt{c}$ 
    and $\rho(\pm (2\sqrt{c},\infty]) = \pm(0,\sqrt{c}]$ respectively.
    In total, $\rho$ maps $\overline{\C} \setminus I_c$ ($I_c = [-2\sqrt{c},2\sqrt{c}]$) onto an open disk $[|z|^2 < c]$ bijectively.
    Note that the imaginary half line $i(0,\infty) = \{iy; y>0\}$ is mapped onto an imaginary segment $i(0,\sqrt{c})$. 

    Likewise, the conjugate root $\rho^*$ maps $\C \setminus I_c$ onto $[|z|^2 > c]$ bijectively so that
    $x \pm i0$ ($x \in I_c^\circ = (-2\sqrt{c},2\sqrt{c})$) constitutes the lower or upper semicircle. 
    
   The semicircular root is also expressed by a continued fraction of constant type: 
\[
S_c(w) =  \cfrac{1}
{-w + \cfrac{-c}
{-w + \cfrac{-c}
{-w + \ddots}
}}.
\]
In fact, $z = c S_c(w)$ satisfies 
\[
  -z = \frac{-c}{-w - z} \iff z^2 + wz + c = 0 
\]
with an asymptotic behavior $z = -c/w + O(1/w^2)$,  
whence it is identified with $\rho$, i.e., $S_c = \rho/c$.

Note that the Stieltjes inversion formula shows that $S_c$ for $c> 0$ is the Stieltjes transform of a semicircle distribution of
density
\[
  \frac{1}{2\pi c} \sqrt{4c - t^2} \quad
  (-2\sqrt{c} \leq t \leq 2\sqrt{c}). 
\]

For later use, we shall give more examples of continued fraction expansions of Stieltjes transforms. 
 Given $0< r < 1$ and a complex parameter $\lambda \not= 0$, consider an analytic function of $w$, 
 \[
   S(w) =  \frac{\lambda - \lambda^{-1}}{2}
         \frac{N(w)}{(1-w^2)(r\lambda + (1-r)\lambda^{-1} - w)}
 \]
 with 
 \[      
   N(w) = 2(\lambda - \lambda^{-1})^{-1}(1-w^2) + (2r-1)w + \sqrt{w^2 - 4r(1-r)}, 
 \]
 which is found in connection with Haagerup's positive definite functions on free groups (\cite{Ya}).

 A bit lengthy but simple computations lead us to the expression 
 \begin{align*}
   S(w) 
        &= \cfrac{1}{\lambda^{-1} - w + \cfrac{-r(1-\lambda^{-2})}{-r\lambda^{-1} - w - r(1-r) S_{r(1-r)}}}\\
   &= \cfrac{1}{\lambda^{-1} - w + \cfrac{-r(1-\lambda^{-2})}{-r\lambda^{-1} - w + \cfrac{-r(1-r)}{-w + \cfrac{-r(1-r)}{-w+\ddots}}}}. 
 \end{align*}

 \section{Shifting Semicircular Roots}
  


 In the context of continued fractions, we define the \textbf{shift} of a function $S(w)$ of a complex variable $w$ by
 \[
   \frac{1}{\alpha - w - \beta S(w)} =
   \begin{pmatrix}
     0 & 1\\
     -\beta & \alpha - w
   \end{pmatrix}. S(w).
 \]
 Here $\alpha$ and $\beta \not= 0$ are complex parameters and a linear fractional transform
 $\displaystyle \frac{az+b}{a'z + b'}$ of $z$ is denoted by
 $\begin{pmatrix}
   a & b\\
   a' & b'
 \end{pmatrix}.z
 $.
 Notice that, if $S$ is the Stieltjes transform of a probability measure on $\R$, so is the shift of $S$ exactly when
 $\alpha \in\R$ and $\beta>0$.

 If $S$ is expressed by a continued fraction like
 \[
   S(w) = \cfrac{1}
{\alpha_0-w + \cfrac{-\beta_0}
{\alpha_{-1} - w + \cfrac{-\beta_{-1}}
{\alpha_{-2} -w + \ddots}
}}, 
 \]
 the shift of $S$ makes the continued fraction one-step longer or higher.

 For $S = 2\rho$ with $\rho$ the semicircular root, one-step shift of $2\rho$ is given by 
 \[
 S_{\alpha,\beta}(w) 
= \frac{1}{\alpha-w - 2\beta\rho}
    = \cfrac{1}
{\alpha-w + \cfrac{-2\beta c}
{-w + \cfrac{-c}
{-w + \ddots}
}}. 
 \]
%

Higher shifts are then obtained by repeating the shift operation:
Given sequences $(\alpha_n)_{n \geq 1}$ and $(\beta_n)_{n \geq 1}$ with $\beta_n \not= 0$, starting from
\[
  \begin{pmatrix}
    A_1 & B_1\\
    C_1 & D_1
  \end{pmatrix}
  =
  \begin{pmatrix}
    0 & 1\\
    -\beta_1 & \alpha_1 - w
  \end{pmatrix}, 
\]
introduce sequences $A_n$, $B_n$, $C_n$ and $D_n$ of polynomials of $w$ by a recurrence relation 
\[
  \begin{pmatrix}
    A_n & B_n\\
    C_n & D_n
  \end{pmatrix}
  =
  \begin{pmatrix}
    0 & 1\\
    -\beta_n & \alpha_n - w
  \end{pmatrix}
  \begin{pmatrix}
    A_{n-1} & B_{n-1}\\
    C_{n-1} & D_{n-1}
  \end{pmatrix} 
\] 
for $n \geq 2$ so that the $(n+1)$-th shift of $2\rho$ is given by
\begin{align*}
    \begin{pmatrix}
    A_n & B_n\\
    C_n & D_n
  \end{pmatrix}.S_{\alpha,\beta}
   &= \begin{pmatrix}
    A_n & B_n\\
    C_n & D_n
  \end{pmatrix}
    \begin{pmatrix}
    0 & 1\\
    -\beta & \alpha - w
  \end{pmatrix}.(2\rho)\\
            &= \begin{pmatrix}
                -\beta B_n & A_n + (\alpha - w)B_n\\
                -\beta D_n & C_n + (\alpha - w) D_n
              \end{pmatrix}.(2\rho)\\
        &= \frac{-2\beta B_n\rho(w) + A_n + (\alpha - w)B_n}{-2\beta D_n\rho(w) + C_n + (\alpha - w)D_n}\\
   &= \frac{F_n + \beta(A_nD_n-B_nC_n)\sqrt{w^2-4c}}{G_n},  
\end{align*}
where 
\begin{align*}
  F_n &= (A_n+\alpha B_n)(C_n + \alpha D_n) + 4c\beta^2 B_n D_n\\
  &\qquad- (1-\beta)w(A_nD_n+ B_nC_n + 2\alpha B_nD_n) + (1-2\beta) w^2 B_n D_n,\\
  G_n &= (C_n + \alpha D_n)^2 + 4c\beta^2 D_n^2 - 2(1-\beta)w (C_n + \alpha D_n)D_n + (1-2\beta) w^2 D_n^2. 
\end{align*}

Note that $A_nD_n - B_nC_n = \beta_1\cdots \beta_n \not= 0$ is a constant as a product of determinants. 


If we set $A_0 = D_0 = 1$ and $B_0 = D_0 = 0$ in the above formula to define $F_0$ and $G_0$, then 
\[
    F_0 = \alpha - (1-\beta)w,\quad
    G_0= \alpha^2 + 4c\beta^2 - 2\alpha(1-\beta)w + (1-2\beta)w^2
  \]
  and 
  \[
  S_{\alpha,\beta}(w)
  = \frac{F_0(w) + \beta\sqrt{w^2 - 4c}}{G_0(w)}. 
\]

From this expression, one sees that 
the initial shifting is rather special in connection with $2\rho$:
  $(\deg F_0,\deg G_0)$ is $(1,2)$ if $(1-\beta)(1-2\beta) \not= 0$,
  whereas   $(0,2)$ for $\beta=1$ and $(1,1)$ for $\beta = 1/2$ but $\alpha \not= 0$. 
  Note that, when $\beta = 1/2$ and $\alpha = 0$, shifting is reduced to repetition:
  $S_{0,1/2} = \rho/c$. 

As for shifting matrices, it is immediate to see that
\[
    \begin{pmatrix}
    \deg A_n & \deg B_n\\
    \deg C_n & \deg D_n
  \end{pmatrix}
  =
  \begin{pmatrix}
    n-2 & n-1\\
    n-1 & n
  \end{pmatrix}, 
\]
which is used, together with $A_n = C_{n-1}$, $B_n = D_{n-1}$ and 
\[
  C_n = (-1)^{n} \beta_1 w^{n-1} + \cdots,
  \quad
  D_n = (-1)^n w^n + \cdots, 
\]
to see
\begin{align*}
  F_n &= \Bigl((1-2\beta)w^2 - 2\alpha(1-\beta)w\Bigr) D_{n-1} D_n + (\text{terms lower than $w^{2n}$}),\\
  G_n &= \Bigl((1-2\beta)w^2 - 2\alpha(1-\beta)w\Bigr) D_n^2 + (\text{terms lower than $w^{2n+1}$}) 
\end{align*}
and then for $n \geq 1$
\begin{align*}
  \deg F_n &= 2n+1,\quad \deg G_n = 2n+2\quad (\beta \not= 1/2), \\
  \deg F_n &= 2n, \quad \deg G_n = 2n+1\quad (\beta = 1/2).
\end{align*}




Now we observe how the roots of the denominator $G_n$ ($n \geq 1$) determines $(n+1)$-th shift parameters
on a level of algebraic freedom.

First let $\deg G_n = 2n+2$. Then the monic part of $G_n$ is specified by $2n+2$ parameters, whereas we have the same number of
shift parameters, i.e., $\alpha,\alpha_1,\cdots,\alpha_n$ and $\beta,\beta_1,\cdots,\beta_n$. 
Thus generically shift parameters are algebraically determined from the information on roots of $G_n$,
which in turn describe $A_n$, $B_n$, $C_n$ and $D_n$.

Next assume $\deg G_n = 2n+1$. Then the monic part of $G_n$ is specified by $2n+1$ parameters, whereas shift parameters are reduced
to $\alpha,\alpha_1,\cdots,\alpha_n$ and $\beta_1,\cdots,\beta_n$ in view of $\beta = 1/2$.
Thus again shift parameters are algebraically determined from the information on roots of $G_n$.

In either case, we can freely choose the monic part of $G_n$ but other parameters including the leading term of $G_n$
as well as $F_n$ and $A_nD_n - B_nC_n$ are determined algebraically.

We shall now look into the above observation for $n=0$ (one-shift) and $n=1$ (two-shift) more closely.

\medskip
\noindent
\section{Algebraic Solutions in Lower Shifts}

\subsection{One-Shift}

First assume that $\beta\not= 1/2$ and $G_0$ is proportional to $(w-a)(w-b)$. 
Then
\[
  \frac{a+b}{2} = \frac{\alpha(1-\beta)}{1-2\beta}, \quad
  ab = \frac{\alpha^2 + 4c\beta^2}{1-2\beta}. 
\]
In terms of a new parameter $\sigma = (1-\beta)/\beta \iff \beta = 1/(\sigma+1)$ ($\beta \not= 1/2 \iff \sigma \not= 1$),
these take the form
\[
  \frac{a+b}{2} = \frac{\sigma}{\sigma-1} \alpha,\quad
  ab = \frac{\alpha^2(\sigma+1)^2 +4c}{\sigma^2-1}, 
\]
which induces an equation on $\sigma$ of the form 
\[
  (a-b)^2 \sigma^4 - 2(a^2 + b^2 - 8c) \sigma^2 + (a+b)^2 = 0. 
\]

For $a+b \not= 0$, any solution of this equation satisfies $\sigma \not= 0$ and determines $\alpha$ and $\beta$ by
\[
  \alpha = \frac{(a+b)(\sigma-1)}{2\sigma}, \quad
  \beta = \frac{1}{\sigma + 1}.
\]
When $a+b = 0$, the equations on $\alpha$ and $\beta$ (or $\sigma$) are reduced to
\[
  \alpha(1-\beta) = 0, \quad
  -a^2 = \frac{\alpha^2 + 4c\beta^2}{1-2\beta}
\]
with solutions given by
\[
  \alpha = 0,\ \beta = \frac{a^2 \pm a \sqrt{a^2-4c}}{4c}
  \quad
  \text{or}
  \quad 
  \alpha = \pm \sqrt{a^2-4c},\ \beta = 1
\]
in such a way that these share a solution if and only if $a^2 = 4c$.

Returning to the equation on $\sigma^2$, it is quadratic if $a \not= b$ and we have 
\begin{align*}
  \sigma^2 &= \frac{a^2+b^2 - 8c \pm \sqrt{(a^2-4c + b^2 - 4c)^2 - (a^2 - 4c + 4c - b^2)^2}}{(a-b)^2}\\
           &= \frac{a^2+b^2 - 8c \pm 2\sqrt{(a^2-4c)(b^2 - 4c)}}{(a-b)^2}\\
           &= \frac{(\sqrt{a^2-4c} \pm \sqrt{b^2-4c})^2}{(a-b)^2}. 
\end{align*}
When $a=b$, the equation is degenerate to
$(a^2-4c)\sigma^2 = a^2$, which has solutions if and only if $a^2 \not= 4c$ in view of $c \not= 0$ with solutions given by
\[
  \sigma = \pm \frac{a}{\sqrt{a^2-4c}}.
\]


Consequently, for $a+b \not= 0$ and $a \not= b$, we have four combinations 
\[
  \sigma =\pm \frac{\sqrt{a^2-4c} \pm \sqrt{b^2-4c}}{a-b}, 
\]
as solutions of $\sigma$, which give rise to four solutions of shift parameters 
\[
  \alpha = \frac{(a+b)(\sigma-1)}{2\sigma},
  \quad
  \beta = \frac{1}{\sigma + 1}
\]
under the proportionality of $G_0$ and $(w-a)(w-b)$,  
the proportionality constant $1-2\beta$ as well as $F_0$ being determined then.

Next let $\beta = 1/2$ and $\alpha \not= 0$. Then $F_0 = \alpha - w/2$ and
$G_0 = c + \alpha^2 - \alpha w$. Thus the monic part of $G_0$ or its root $\alpha + c/\alpha$ determines $\alpha$ and then $F_0$.

As a summary, we have the following. 

\begin{Proposition}\label{one-shift}
  One-shift condition on $(F_0 + \beta\sqrt{w^2-4c})/G_0$ determines the multiplying constant $1-2\beta$ as well as $F_0$ from the roots of $G_0$:
  \begin{enumerate}
    \item
      If $\deg G_0 = 2$ and $G_0$ is proportional to $(w-a)(w-b)$ ($a \not= b$), 
      then we have (possibly degenerate) four choices of $(\alpha,\beta)$ with the limit solutions for $a+b = 0$ expressed by 
      \[
        (\alpha,\beta) = (\pm \sqrt{a^2 -4c},1)\ \text{or}\ (0,a/(a \pm \sqrt{a^2-4c})) 
      \]
      including the degenerate unique solution $(\alpha,\beta) = (0,1)$ for $b = -a = \pm 2\sqrt{c}$. 
    \item
      If $\deg G_0 = 2$ and $G_0$ is proportional to $(w-a)^2$, 
      then
      \[
        \alpha = a \pm \sqrt{a^2-4c}, \quad \beta = \frac{4c-a^2 \mp a\sqrt{a^2-4c}}{4c}
        \]
      (we have two solustions for $a^2 \not= 4c$, whereas $\beta = 0$ for $a^2 = 4c$). 
    \item
      If $\deg G_0 = 1$, then $\beta = 1/2$ and
      $\alpha + c/\alpha$ is the root of $G_0$, which determines $\alpha$ and then $F_0$ up to an exchange $\alpha \leftrightarrow c/\alpha$. 
    \end{enumerate}
  \end{Proposition}

\begin{Remark}
The case $\alpha=0$ in (i) corresponds to the classical Kesten measure and the case (iv) covers the Marchenko-Pastur law. 
\end{Remark}


\subsection{Two-Shift}

Now we shift $S_{\alpha,\beta}$ one-step further to get 
\[
  S(w) = \begin{pmatrix} 0 & 1\\ -\delta & \gamma - w\end{pmatrix}.S_{\alpha,\beta} = \frac{1}{\gamma - w - \delta S_{\alpha,\beta}(w)}
  = \frac{F + \beta\delta \sqrt{w^2-4c}}{G}, 
\]
where 
\begin{align*}
  F &= \alpha(-\delta + \alpha\gamma - \alpha w) + 4c\beta^2 (\gamma - w)\\
      &\qquad- (1-\beta) w (-\delta + 2\alpha(\gamma - w)) + (1-2\beta)w^2(\gamma - w)\\
    &= -(1-2\beta) w^3 + \Bigl(2\alpha(1-\beta) + (1-2\beta)\gamma\Bigr) w^2\\
    &\qquad- \Bigl(\alpha^2 + 4c\beta^2 + (1-\beta)(-\delta + 2\alpha\gamma)\Bigr) w
      + \alpha^2\gamma - \alpha\delta + 4c\beta^2\gamma
\end{align*}
and, with the parameter switched to $\Delta = -\delta + \alpha\gamma$ from $\delta$, 
\begin{align*}
  G &= (-\delta + \alpha\gamma -\alpha w)^2 + 4c\beta^2(\gamma - w)^2\\
  &\qquad- 2(1-\beta)w(-\delta + \alpha\gamma - \alpha w)(\gamma - w) + (1-2\beta)w^2(\gamma - w)^2\\
    &= (\Delta -\alpha w)^2 + 4c\beta^2(\gamma - w)^2 - 2(1-\beta)w(\Delta - \alpha w)(\gamma - w)\\
  &\qquad+ (1-2\beta)w^2(\gamma - w)^2\\
    &= (1-2\beta)w^4 - 2\Bigl(\alpha(1-\beta) + (1-2\beta)\gamma\Bigr)w^3\\
      &\qquad+ \Bigl(\alpha^2 + 4c\beta^2 + (1-2\beta)\gamma^2 + 2(1-\beta)(\Delta + \alpha\gamma)\Bigr) w^2\\
      &\qquad\qquad- 2\Bigl(4c\beta^2\gamma + \Delta(\alpha + (1-\beta)\gamma)\Bigr) w + \Delta^2 + 4c\beta^2\gamma^2. 
\end{align*}

When $\beta = 1/2$, these are simplified to
\begin{align*}
  F &= \alpha w^2 - (c+\alpha^2 +\alpha\gamma - \delta/2)w + \alpha(-\delta + \alpha\gamma) + c\gamma,\\
  G &= -\alpha w^3 + (c+\alpha^2 + \alpha\gamma + \Delta)w^2
      - (2c\gamma + \Delta (2\alpha + \gamma)) w + c\gamma^2 + \Delta^2. 
\end{align*}
Even in that case, there is no simple way to extract solutions from the general monic part of $G$ (or from the information on roots of $G$). 

\begin{Example}\label{degree4}
  Assume that $G$ is an even polynomial of degree $4$ ($\beta \not= 1/2$ particularly) satisfying $G(\pm \zeta) = 0$
  (more precisely, $G(w)$ being factorized through $w^2-\zeta^2$). Then
  \[
    \alpha(1-\beta) + (1-2\beta)\gamma = 0 = 
    4c\beta^2\gamma + \Delta(\alpha + (1-\beta)\gamma), 
  \]
  which is equivalent to
  \[
    \gamma = - \frac{1-\beta}{1-2\beta} \alpha,
    \quad
    \alpha(\Delta + 4c(1-\beta)) = 0, 
  \]
  and
  \begin{multline*}
    (1-2\beta)\zeta^4 + 4c\beta^2\zeta^2 + 2(1-\beta)\Delta\zeta^2 + \Delta^2\\
    + \alpha^2\zeta^2 + (1-2\beta)\gamma^2\zeta^2 + 2(1-\beta) \alpha\gamma\zeta^2 + 4c\beta^2 \gamma^2 = 0.
  \end{multline*}
 \begin{enumerate} 
\item When $\alpha = 0$,
  we have $\gamma = 0$ and
  \[
    G = (1-2\beta)w^4 + 2(2c\beta^2 + (1-\beta)\Delta) w^2 + \Delta^2
  \]
  with the condition $G(\pm \zeta) = 0$ expressed by 
  \begin{gather*}
    0 = (1 - 2\beta)\zeta^4 + 4c\beta^2\zeta^2 + 2(1-\beta)\zeta^2\Delta + \Delta^2\\
    \iff \Delta = -(1 - \beta)\zeta^2 \pm \beta\zeta \sqrt{\zeta^2-4c},
  \end{gather*}
  which is used to have 
  \begin{align*}
    F &= -(1-2\beta) w^3 -(4c\beta^2 + (1-\beta)\Delta)w,\\
    G &= (1-2\beta)(w^4-\zeta^4) + 2(2c\beta^2 + (1-\beta)\Delta) (w^2-\zeta^2). 
  \end{align*}
  In this way, taking $\beta \not= 1/2$ as a free parameter, other parameters are algebraically expressed by $\beta$.
  
 \item  When $\alpha \not= 0$, $\Delta = -4c(1-\beta)$ together with $\gamma = -\alpha(1-\beta)/(1-2\beta)$
  is used in the condition $G(\pm \zeta) = 0$ to have
  \[
    \Bigl(\beta^2\zeta^2 + (4c-\zeta^2)(1-\beta)^2\Bigr) \Bigl((4c-\zeta^2)(1-2\beta)^2 + \alpha^2\beta^2\Bigr) = 0. 
  \]
  Thus, if 
  \[
    \beta^2\zeta^2 + (4c-\zeta^2)(1-\beta)^2 = 0 \iff \beta = -\frac{\zeta^2-4c \pm \zeta\sqrt{\zeta^2-4c}}{4c},  
  \]
  we can choose $\alpha$ as a free parameter so that $\gamma = -\alpha(1-\beta)/(1-2\beta)$ and 
  $\delta = -\Delta + \alpha\gamma$ with $\beta$ and $\Delta$ algebraically determined from $\zeta$.
  In terms of these, we have
  \begin{align*}
    F(w) &= -(1-2\beta)w^3 + \alpha(1-\beta)w^2\\
    &\qquad\qquad+ \left( \frac{\alpha^2\beta^2}{1-2\beta} + 4c(1-2\beta)\right) w
           - 4c\frac{(1-\beta)^3}{1-2\beta}\alpha,\\
    G(w) &= (1-2\beta) (w^4-\zeta^4)
           - \left( \frac{\alpha^2\beta^2}{1-2\beta} - 4c\beta^2 + 8c(1-\beta)^2\right) (w^2 - \zeta^2). 
  \end{align*}

  
  Otherwise, we take $\beta$ as a free parameter to have 
  \[
    \alpha = \pm\sqrt{\zeta^2-4c} \frac{1-2\beta}{\beta}, \quad
    \gamma = \mp\sqrt{\zeta^2-4c} \frac{1-\beta}{\beta}, \quad
    \Delta = -4c(1-\beta)
  \]
  and
  \begin{align*}
    F &= -(1-2\beta)w^3 \pm \sqrt{\zeta^2-4c} \frac{(1-\beta)(1-2\beta)}{\beta} w^2\\
        &\qquad\qquad\qquad+ (1-2\beta)\zeta^2 w \mp 4c \sqrt{\zeta^2 - 4c} \frac{(1-\beta)^3}{\beta},\\
  G  &= \Bigl((1-2\beta)w^2 - 4c(1-\beta)^2\Bigr) (w^2-\zeta^2). 
  \end{align*}
\end{enumerate}
\end{Example}

\begin{Example}\label{degree3}
  Let $G$ be of degree $3$ ($\beta = 1/2$, $\alpha \not= 0$) and assume that $G(\pm \zeta) = 0$ (more precisely,
  $G(w)$ is factored through $w^2-\zeta^2$): 
  \[
    \alpha\zeta^2 + 2c\gamma + \Delta(2\alpha + \gamma)
    = 0 = (c + \alpha^2 + \alpha\gamma + \Delta)\zeta^2 + c\gamma^2 + \Delta^2.
  \]
  We rewrite these relations as 
  \[
    \left(\alpha + \frac{\gamma}{2}\right) \left( \Delta + \frac{\zeta^2}{2}\right) = \left(\frac{\zeta^2}{4} - c\right) \gamma
  \]
 and  
   \[
    \zeta^2 \left( \alpha + \frac{\gamma}{2} \right)^2 + \left(\Delta + \frac{\zeta^2}{2}\right)^2
    = \left(\frac{\zeta^2}{4} - c\right) (\gamma^2 + \zeta^2). 
  \]
  
  To get algebraic solutions of parameters, we first deal with a somewhat exceptional case, namely 
  \[
    \left(\alpha + \frac{\gamma}{2}\right) \left( \Delta + \frac{\zeta^2}{2}\right) = 0
    \iff \left(\frac{\zeta^2}{4} - c\right) \gamma = 0
  \]
  in the first equation. 
  If $\gamma = 0$, we then have $\Delta + \frac{\zeta^2}{2} = 0$ in view of $\alpha \not = 0$ and
  the exisitence of non-zero $\delta = - \Delta + \alpha \gamma = \zeta^2/2$ requires $\zeta \not= 0$,
  which is used in the second equation to have $\alpha^2 = (\zeta^2 - 4c)/4$. Thus the equations are solved to be 
  \[
    \alpha = \pm \frac{1}{2}\sqrt{\zeta^2 - 4c},\quad \gamma = 0,\quad \delta = \frac{\zeta^2}{2}
  \]
  so that 
  \[
    G(w) = -\alpha\left(w - \frac{\zeta^2}{4\alpha}\right) (w^2-\zeta^2)
  \]

  Otherwise, $\zeta^2 = 4c \not= 0$ and the second equation is reduced to
  \[
      \zeta^2 \left( \alpha + \frac{\gamma}{2} \right)^2 + \left(\Delta + \frac{\zeta^2}{2}\right)^2 = 0, 
  \]
  which is combined with  $\left(\alpha + \frac{\gamma}{2}\right) \left( \Delta + \frac{\zeta^2}{2}\right) = 0$ to have 
  $\alpha + \frac{\gamma}{2} = 0 = \Delta + \frac{\zeta^2}{2}$ 
  and then
  $\delta = - \Delta + \alpha\gamma = 2(c-\alpha^2)$.
We are thus lead to a solution for $\zeta^2 = 4c$ having $\alpha \not= 0$ as a free parameter and 
  \[
    \gamma = -2\alpha, \quad \delta = 2(c-\alpha^2) 
  \]
  so that
  \[
    G(w) = -\alpha\left(w - \alpha - \frac{c}{\alpha} \right) (w^2 - 4c).
  \]

  We also separate one more exceptional case $\zeta = 0$;
  \[
    \left(\alpha + \frac{\gamma}{2}\right) \Delta = -c\gamma,
    \quad
    \Delta^2 = -c\gamma^2.
  \]
  Since the case $(\alpha + \gamma/2)\Delta = 0 \iff \gamma = 0$ is already covered above, let $\gamma \not= 0$
  (necessarily $\Delta \not = 0$). We then have
  \[
    \left(\alpha + \frac{\gamma}{2}\right) \gamma = \Delta \iff \delta = - \frac{1}{2} \gamma^2
  \]
  and the equations are reduced to
  $\displaystyle\left(\alpha + \frac{\gamma}{2}\right)^2 = -c$ so that one degree of free parameters remains and
  \[
     G(w) = -\alpha\left(w - \alpha - \frac{c}{\alpha} \right) w^2.
  \]
  
  We now assume that $\zeta(\zeta^2 - 4c)\gamma \not= 0$ and
  take the quotient of the second equation by the first one to get
%
  \[
    \zeta\frac{\alpha + \gamma/2}{\Delta + \zeta^2/2} + \frac{1}{\zeta} \frac{\Delta + \zeta^2/2}{\alpha + \gamma/2}
    = \frac{\gamma}{\zeta} + \frac{\zeta}{\gamma}, 
  \]
  whence a new variable $\lambda = (\Delta + \zeta^2/2)/(\alpha + \gamma/2)$ satisfies $\lambda = \gamma$ or $\zeta^2/\gamma$.

  In accordance with this, we write $c = \zeta^2 r(1-r)$ so that 
  \[
    \left( \alpha + \frac{\gamma}{2}\right)^2 = \left( \frac{\zeta^2}{4} - c\right) \frac{\gamma}{\lambda}
  =
  \begin{cases}
    \zeta^2(r-1/2)^2 &(\gamma = \lambda)\\
    \zeta^4(r-1/2)^2/\lambda^2 &(\gamma = \zeta^2/\lambda)
  \end{cases}. 
  \]
  
  All shifting parameters (including $\beta = 1/2$) are then expressed in terms of $\lambda$ ($r$ and $\zeta$ being constants) as follows:
\begin{enumerate}
\item
  For $\gamma = \lambda$, 
  \[
    \alpha = -\frac{\lambda}{2} \pm \zeta(r - \frac{1}{2}),\quad
    \Delta = - \frac{\zeta^2}{2} \pm \zeta(r- \frac{1}{2}) \lambda,\quad 
    \delta = -\Delta + \alpha\gamma = \frac{\zeta^2 - \lambda^2}{2}
  \]
  ($\pm$ corresponds to the symmetry $r \leftrightarrow 1-r$ and it suffices to consider one choice) and 
   \[
    G(w) = -\alpha\left( w + \alpha + \gamma + \frac{c+\Delta}{\alpha}\right) (w^2 - \zeta^2).
  \]
  Thus the parameter $\lambda$ is determined from the remaining root
  $-(\alpha+\gamma) - (c+\Delta)/\alpha$ of $G$ by solving a quadratic equation. 
 \item
  For $\gamma = \zeta^2/\lambda$, again $\pm(r -1/2)$ corresponds to the symmetry $r \leftrightarrow 1-r$ and we may choose
  \[
    \alpha = -\zeta^2 \frac{r}{\lambda}, \quad
    \Delta = - r\zeta^2, \quad 
    \delta = r\zeta^2 \frac{\lambda^2 - \zeta^2}{\lambda^2}, 
  \]
  which are used to obtain 
  \[
    G(w) 
    = \frac{r\zeta^2}{\lambda} \left(w - r\lambda - \zeta^2(1-r)\frac{1}{\lambda}\right) (w^2 - \zeta^2). 
  \]
  %
\end{enumerate}
\end{Example}


\section{Positive Solutions}
We here discuss positivity of algebraic solutions described in the previous section.
\subsection{One-Shift}
First consider positive solutions of $S_{\alpha,\beta}$ in Proposition~\ref{one-shift}, i.e.,
solutions satisfying the condition $\alpha \in \R$ and $\beta > 0 \iff \sigma = (1-\beta)/\beta > -1$,
which will be described in terms of roots of $G_0$. 

In the degenerate case (iii) ($\beta = 1/2$), the positivity is nothing but the reality $\alpha \in \R$ and equivalent to the condition
that the root of $G_0$ is in the range $\R \setminus (-2c,2c)$. 

We then look into the case $\beta \not= 1/2 \iff \sigma \not= 1$, i.e., the case $\deg G_0 = 2$.
Since the reality of $S_{\alpha,\beta}$ is equivalent to that of $F_0$ and $G_0$, the roots $a,b$ of $G_0$ should satisfy $a+b, ab \in \R$.

The double root case (ii) ($a=b \in \R$) 
admits real solutions if and only if $a^2 \geq 4c$. 
Among these, a positive solution appears if and only if $a^2 > 4c$ and it takes the form of 
$\beta = (4c - a^2 + |a| \sqrt{a^2-4c})/4c$ and
\[
  \alpha =
  \begin{cases}
    a - \sqrt{a^2-4c} &(a>2\sqrt{c}), \\
    a + \sqrt{a^2-4c} &(a < -2\sqrt{c}).
  \end{cases} 
\]

We next discuss the case (i) ($a \not= b$). When $a+b=0$ and $0 \not= a^2 \in \R$,
the solution $(\alpha,\beta) = (\pm \sqrt{a^2-4c}, 1)$ is real if and only if $a^2 \geq 4c$
with two real solutions automatically positive because $\beta = 1 > 0$.
Meanwhile other solutions $\alpha=0$ and $\beta = a/(a\pm \sqrt{a^2-4c})$ are real if and only if
$a^2 \geq 4c$ or $a^2 < 0$. If $a^2 \geq 4c$, $a/(a\pm \sqrt{a^2-4c}) >0$ and two real solutions are positive.
For a purely imaginary root $a = it$ with $0 \not= t \in \R$, a positive choice in 
\[
  \beta = \frac{a}{a \pm \sqrt{a^2-4c}} = \frac{t}{t \pm \sqrt{t^2+4c}}
\]
is given by
\[
  4c\beta =
  \begin{cases}
    -t^2 + t\sqrt{t^2+4c} &(t>0),\\
    -t^2 - t\sqrt{t^2+4c} &(t<0). 
  \end{cases}
\]

Now consider the case (i) ($a \not= b$) and assume that $a+b \not= 0$.
Then $\sigma \not= 0$ and we have two alternatives: $a,b \in \R$ or $a = \overline{b} \not\in \R$.
%

For real roots $a,b \in \R$, the expression 
\[
  \sigma =\pm \frac{\sqrt{a^2-4c} \pm \sqrt{b^2-4c}}{a-b}, 
\]
admits a real combination of $\sigma$ if and only if $a^2-4c \geq 0$ and $b^2 - 4c \geq 0$. 

For imaginary roots, putting $a=\zeta$ and $b = \overline{\zeta}$ with $\zeta \in \C \setminus \R$,
we have 
 \[
   \sigma 
   = \pm i \frac{\sqrt{4c - \zeta^2} \pm \overline{\sqrt{4c - \zeta^2}}}{\zeta - \overline{\zeta}}. 
 \]
 Among these, real ones are given by 
 \[
   \sigma = \pm i \frac{\sqrt{4c - \zeta^2} + \overline{\sqrt{4c - \zeta^2}}}{\zeta - \overline{\zeta}}
   = \pm \frac{\text{Re}\,\sqrt{4c - \zeta^2}}{\text{Im}\, \zeta}. 
 \] 




 \begin{Lemma}
   For real combinations of $\sigma$, the following holds.
   \begin{enumerate}
   \item
     If $a,b \in \R$ satisfy $a^2 \geq 4c$ and $b^2 \geq 4c$ with the case $a=b = \pm 2\sqrt{c}$ excluded, 
     then $|\sigma| \geq 1$ or $|\sigma| \leq 1$ according to
     $ab > 0$ or $ab < 0$.
   \item
     If $a = \overline{b} \in \C \setminus \R$, then $|\sigma| > 1$. 
   \end{enumerate}
 \end{Lemma} 

 \begin{proof}
  (i) Starting with $(a^2-4c)(b^2-4c) \leq (ab-4c)^2$ 
  and the assumption $|ab| \geq 4c$, if $ab>0$,
  $ab-4c = |ab-4c| \geq \pm \sqrt{a^2-4c}\sqrt{b^2-4c}$ implies
  $ab-4c \mp \sqrt{a^2-4c}\sqrt{b^2-4c} \geq 0$ and hence 
  \[
    (\sqrt{a^2-4c} \pm \sqrt{b^2-4c})^2 - (a-b)^2 = 2ab - 8c \pm 2\sqrt{a^2-4c}\sqrt{b^2-4c} \geq 0. 
  \]
  Likewise, if $ab < 0$, $4c-ab = |ab-4c| \geq \pm \sqrt{a^2-4c}\sqrt{b^2-4c}$ implies
  $ab-4c \pm \sqrt{a^2-4c}\sqrt{b^2-4c} \leq 0$ and hence 
  \[
    (\sqrt{a^2-4c} \pm \sqrt{b^2-4c})^2 - (a-b)^2 = 2ab - 8c \pm 2\sqrt{a^2-4c}\sqrt{b^2-4c} \leq 0. 
  \]
  
 (ii) 
 Let $\zeta = \xi + i\eta$ with $\xi, \eta \in \R$ and $\eta \not= 0$ so that 
 \[
   \sigma = \pm \text{Re}\,\sqrt{1 + (4c-\xi^2)/\eta^2 - 2i\xi/\eta}.
 \]
 To get a lower bound of $|\sigma|$, we write
 \[
   \sqrt{1 + \frac{4c - \xi^2}{\eta^2} - 2i \frac{\xi}{\eta}} = x + iy \iff
   \begin{cases}
     x^2 - y^2 &= 1 + (4c-\xi^2)/\eta^2,\\
     xy &= -\xi/\eta, 
   \end{cases}
 \]
 which is solved to be
 \[
   x^2 = \frac{t + \sqrt{t^2 + 4\xi^2/\eta^2}}{2},\quad
   y^2 = \frac{-t + \sqrt{t^2 + 4\xi^2/\eta^2}}{2}
 \]
 with the choice $t = 1 + (4c-\xi^2)/\eta^2$.
 Since $x^2$ is strictly increasing as a function of $t \in \R$ and takes $1$ at $t = (\eta^2-\xi^2)/\eta^2$,
 we see that $\sigma^2 = x^2 > 1$.
\end{proof}

\begin{Remark}
 $y^2 \geq 1$ or $y^2 \leq 1$ according to $\xi^2 - \eta^2 \geq 2c$ or $\xi^2 - \eta^2 \leq 2c$. 
 \end{Remark}

 \begin{Corollary} Let $\deg G_0 = 2$, i.e., $\beta \not= 1/2$. 
   \begin{enumerate}
     \item
   In the real root case, $\sigma > -1$ is satisfied by exactly two positive solutions of $\sigma$ if $ab>0$, whereas
   all four solutions of $\sigma$ satisfy $\sigma > -1$ if $ab < 0$.
 \item
   In the imaginary root case, given $\zeta \in \C \setminus \R$ satisfying $G_0(\zeta) = 0$, there exists exactly one solution $\sigma$ 
   satisfying $\sigma > -1$. 
 \end{enumerate}
\end{Corollary}

\subsection{Two-Shift}
Now we go on to two-shifts and consider the situation in Example~\ref{degree4}.
Recall then that $\alpha = 0$ or $\Delta = -4c(1-\beta)$. 
Notice also that, to have a solution of real parameters, $G$ is necessarily real and $\zeta^2 \in \R$ particularly.
\begin{enumerate}
\item
$\alpha = 0$ with $\beta \not= 1/2$ a free parameter.
  The positivity condition on Jacobi parameters is then given by $\Delta = -\delta < 0$ for some choice of $\beta > 0$.
  The existence of positive $\delta$ therefore requires $\zeta \not= 0$ and,
  for a real $\zeta$, we may suppose that $\zeta = s > 0$ with $s^2 - 4c \geq 0$.
  The condition $\delta > 0$ is then equivalent to
\[
  \frac{1-\beta}{\beta} > \pm \frac{\sqrt{s^2 - 4c}}{s}. 
\]
Thus, for a $\beta$ in the range $(1-\beta)/\beta > \sqrt{s^2-4c}/s$, two choices of $\delta$ satisfy $\delta > 0$, 
for the range $-\sqrt{s^2-4c}/s < (1-\beta)/\beta \leq \sqrt{s^2 - 4c}/s$,
a positive $\delta$ is unique and given by 
$\delta = (1-\beta) s^2 + \beta s \sqrt{s^2 - 4c} > 0$,
and no positive $\delta$'s for the range $(1-\beta)/\beta \leq -\sqrt{s^2-4c}/s$. 

For an imaginary $\zeta$, we may take $\zeta = is$ with $s>0$. Then 
\[
  (1-\beta)\zeta^2 \pm \beta\zeta \sqrt{\zeta^2 - 4c} > 0
  \iff
  \frac{1-\beta}{\beta} < \pm \sqrt{1+ 4c/s^2} 
\]
shows that the correct choice is
\[
  \delta =(1-\beta)\zeta^2 \pm \beta \zeta\sqrt{\zeta^2 - 4c} = -(1-\beta)s^2 + \beta s\sqrt{s^2+4c}
\]
and the condition $\delta > 0$ is satisfied exactly when 
\[
  \beta > \frac{s\sqrt{s^2 + 4c} - s^2}{4c}.
\]

\item
  $\Delta = -4c(1-\beta)$ and $\alpha \not= 0$. We then have two subcases:
  \begin{enumerate}
    \item
$\alpha\not=0$ is a free parameter and (ii)-2 $\beta \not= 1/2$ is a free parameter with
the other parameters algebraically expressed in terms of $\beta$.

(ii)-1: 
$\beta$ is specified by the equation 
\[
  \frac{(1-\beta)^2}{\beta^2} = \frac{\zeta^2}{\zeta^2 - 4c}. 
\]

If $\zeta = s\geq 0$, $s$ must satisfy $s^2 - 4c > 0$, i.e., $s>2\sqrt{c}$, to get a real $\beta \not= 0$.
Moreover, in view of $(1-\beta)^2/\beta^2 > 1$,
there is a unique positive $\beta$ satisfying $(1-\beta)/\beta = s/\sqrt{s^2-4c} > 1$, i.e.,
$\beta = \sqrt{s^2-4c}/(s + \sqrt{s^2 - 4c}) < 1/2$. Then 
\[
  \gamma = -\alpha\frac{1-\beta}{1-2\beta} = -\alpha\frac{s}{s- \sqrt{s^2 - 4c}},
  \quad
  \Delta = -4c\frac{s}{s+ \sqrt{s^2-4c}}
\]
and
\[
  \delta = -\Delta + \alpha\gamma = -\frac{(\alpha^2-4c)s^2 + (\alpha^2+4c)s\sqrt{s^2-4c}}{4c}. 
\]
Thus $\delta > 0$ if and only if
\[
  (\alpha^2+4c) \sqrt{s^2-4c} < (4c-\alpha^2)s
\]
which is equivalent to
\[
  \alpha^2 < 4c\ \text{and}\ 
  \alpha^2 < 4c \frac{s-\sqrt{s^2-4c}}{s+\sqrt{s^2-4c}} = (s-\sqrt{s^2-4c})^2, 
\]
namely
\[
  \alpha^2 < (s-\sqrt{s^2-4c})^2 \iff   -s + \sqrt{s^2-4c} < \alpha < s - \sqrt{s^2-4c}.
\]

Now let $\zeta = is$ with $s>0$. This time,
\[
  \frac{(1-\beta)^2}{\beta^2} = \frac{s^2}{s^2+4c} \in (0,1)
\]
and we have two positive $\beta$'s belonging to $(1/2,1) \sqcup (1,\infty)$ so that 
\[
  \frac{1-\beta}{\beta} = \pm \frac{s}{\sqrt{s^2+4c}}, \quad
  \gamma 
  = -\alpha\frac{s}{s\mp \sqrt{s^2 + 4c}},
    \quad
  \Delta = -4c\frac{s}{s \pm \sqrt{s^2 + 4c}}
\]
and
\[
  \delta = -\Delta + \alpha\gamma = -\frac{s}{4c} \Bigl((4c-\alpha^2)s \mp (4c+\alpha^2)\sqrt{s^2+4c}\Bigr). 
\]
In view of 
\[
  (4c-\alpha^2)s - (4c+\alpha^2) \sqrt{s^2+4c} < 0 <   (4c-\alpha^2)s + (4c+\alpha^2) \sqrt{s^2+4c}, 
\]
the correct choice of positivity is then 
\[
  \frac{1-\beta}{\beta} = \frac{s}{\sqrt{s^2+4c}}, \quad
  \gamma = -\alpha\frac{s}{s - \sqrt{s^2 + 4c}},
    \quad
  \Delta = -4c\frac{s}{s + \sqrt{s^2 + 4c}}
\]
and
\[
  \delta = \frac{s}{4c} \Bigl((\alpha^2-4c)s + (4c+\alpha^2)\sqrt{s^2+4c}\Bigr).
\]
with $\alpha$ an arbitrary non-zero real parameter. 





\item 
From algebraic expressions of $\alpha$, $\gamma$ and $\delta$ (or $\Delta$) in Example~`ref{degree4},
  their reality is fulfilled if and only if $\zeta^2 - 4c \geq 0$ and $0 \not= \beta \in \R$.
  Then, putting $\zeta = s > 0$ with $s^2 \geq 4c$, we have 
  \[
    \delta = - \Delta + \alpha\gamma = \frac{(1-\beta)}{\beta^2} \Bigl(s^2\beta^2 - (s^2 - 4c)(1-\beta)^2\Bigr)  
  \]
and the condition $\delta > 0$ is rephrased by
  \[
    (1-\beta) \bigl(\sqrt{s^2-4c} + (s - \sqrt{s^2-4c})\beta\bigr) \bigl(-\sqrt{s^2-4c} + (s + \sqrt{s^2-4c})\beta\bigr) > 0. 
  \]
  In view of
  \[
-\frac{\sqrt{s^2-4c}}{s - \sqrt{s^2-4c}} < 0 < \frac{\sqrt{s^2-4c}}{s + \sqrt{s^2-4c}} < 1, 
  \]
  the positivity condition $\beta > 0$ and $\delta > 0$ is equivalent to
  \[
    \frac{\sqrt{s^2-4c}}{s + \sqrt{s^2-4c}} < \beta < 1.
  \]
  Note that $\sqrt{s^2-4c}/(s+\sqrt{s^2-4c}) < 1/2$ ($s>0$).
\end{enumerate}
\end{enumerate}

\bigskip
We now investigate the positivity condition on parameters in Example~\ref{degree3}.
First notice that, for $\zeta = 0$, $c\gamma^2 + \Delta^2 = 0$ implies $\gamma = \Delta = 0$ if parameters are real, whence
$\delta = -\Delta + \alpha\gamma = 0$. Thus, to have a solution $\delta > 0$, $\zeta = 0$ should be excluded.
As for the case $(\zeta^2-4c)\gamma = 0$,
the $\gamma=0$ solution is positive if and only if $\zeta^2 > 4c$, whereas the $\zeta^2=4c$ solution is positive
if and only if $0 < \alpha^2 < c$. 

Next we go on to the $\lambda$-solutions. If there exist any real solutions,
we should have $\zeta^2, \gamma \in \R$ and then $\lambda \in \R$ as their combinations.


Recall taht, with $\lambda$ a free real parameter, we have two analytic expressions for parameters:
\[
 \gamma = \lambda, \quad \alpha = -\frac{\lambda}{2} \pm \zeta ( r- \frac{1}{2}),\quad
  \Delta = - \frac{\zeta^2}{2} \pm \zeta(r - \frac{1}{2}) \lambda,\quad 
  \delta = \frac{\zeta^2 - \lambda^2}{2}
\]
  \[
    \gamma = \frac{\zeta^2}{\lambda}, \quad
    \alpha = -\zeta^2 \frac{r}{\lambda}, \quad
    \Delta = - r\zeta^2, \quad 
    \delta = r\zeta^2 \frac{\lambda^2 - \zeta^2}{\lambda^2}. 
  \]

  Consider the case $\zeta = s > 0$. To have real solutions, one of $\zeta(r-1/2)$ and $r/\lambda$ must be real, i.e., $r \in \R$,
  which together with $r(1-r) = c/\zeta^2 > 0$ implies $0 < r < 1$ and hence $s^2 - 4c = 4s^2(r-1/2)^2 \geq 0$.
  Conversely, if $\zeta = s >0$ with $s^2 \geq 4c$, then
  \[
    r = \frac{s \pm \sqrt{s^2 - 4c}}{2s}
  \]
  satisfies $0 < r < 1$.
  Thus two analytic expressions of $\alpha$, $\gamma$ and $\delta$ give real solustions for $\lambda \in \R$
  with the positivity condition $\delta > 0$ selecting one of them according to $\lambda^2 - s^2 > 0$ or $\lambda^2 - s^2 < 0$.


  Next let $\zeta = is$ with $s>0$. From the expression 
  \[
    r = \frac{1 \pm \sqrt{1 + 4c/s^2}}{2}, 
  \]
  $r$ is real and one sees that real solutions are restricted to the one
  \[
    \gamma = - \frac{s^2}{\lambda}, \quad 
    \alpha = s^2 \frac{r}{\lambda}, \quad
    \Delta = rs^2, \quad 
    \delta = -rs^2 \frac{s^2 + \lambda^2}{\lambda^2}. 
  \]
  The positivity $\delta > 0$ is then satisfied for the choice
  \[
   r = \frac{1 - \sqrt{1 + 4c/s^2}}{2} < 0. 
  \] 

\section{Residue Calculus in Stieltjes Transforms}

  



Let $\mu(dt) = \Bigl(\sqrt{4c-t^2}/Q(t)\Bigr)dt$ be a finite measure supported by $I_c = [-2\sqrt{c},2\sqrt{c}]$,
where $Q(t)$ is a rational function of $t$
satisfying $Q(t) > 0$ ($-2\sqrt{c} < t < 2\sqrt{c}$) ($Q$ being a ratio  of real polynomials then)
with $\pm 2\sqrt{c}$ being at worst simple poles of $1/Q$ 
($(w^2 - 4c)/Q(w)$ being bounded near $w=\pm 2 \sqrt{c}$. 
A typical example is the density function $1/\sqrt{4c - t^2}$ ($Q(w) = 4c - w^2$) of arcsine law.

We shall here review the process to compute the Stieltjes transform of $\mu$ as a contour integral. 
Let $C_\epsilon$ be a contour surrounding the interval $I_c$ counter-clockwise and shrinking to the line segment $I_c$ as $\epsilon \to +0$. 
Then, given $w \in \C \setminus \R$, the contour integral
\[
  \oint_{C_\epsilon} \frac{1}{z-w} \frac{\sqrt{z^2 - 4c}}{Q(z)}\, dz
\]
($\sqrt{z^2-4c}$ being analytic on $\C \setminus I_c$) is approximated by
\begin{align*}
  &\int_{-2\sqrt{c}}^{2\sqrt{c}} \frac{1}{x-i\epsilon - w} \frac{\sqrt{(x-i\epsilon)^2 - 4c}}{Q(x-i\epsilon)}\, dx
      - \int_{-2\sqrt{c}}^{2\sqrt{c}} \frac{1}{x+i\epsilon - w} \frac{\sqrt{(x+i\epsilon)^2 - 4c}}{Q(x+i\epsilon)}\, dx\\
  &= \int_{-2\sqrt{c}}^{2\sqrt{c}} \frac{1}{x-i\epsilon - w} \frac{-i\sqrt{4c-(x-i\epsilon)^2}}{Q(x-i\epsilon)}\, dx
      - \int_{-2\sqrt{c}}^{2\sqrt{c}} \frac{1}{x+i\epsilon - w} \frac{i\sqrt{4c-(x+i\epsilon)^2}}{Q(x+i\epsilon)}\, dx, 
\end{align*}
which approaches
\[
      -2i\int_{-2\sqrt{c}}^{2\sqrt{c}} \frac{1}{x - w} \frac{\sqrt{4c-x^2}}{Q(x)}\, dx = -2i S_\mu(w) 
\]
as $\epsilon \to +0$.
Thus, for a sufficiently small $\epsilon>0$, 
\[
S_\mu(w) = \frac{i}{2} \oint_{C_\epsilon} \frac{1}{z-w} \frac{\sqrt{z^2 - 4c}}{Q(z)}\, dz, 
\]
whereas, for a sufficiently large $R>0$ and a sufficiently small $r>0$, the residue theorem gives
\begin{align*}
  \oint_{|z| = R} &\frac{1}{z-w} \frac{\sqrt{z^2 - 4c}}{Q(z)}\, dz\\
  &= \sum_{j=0}^l \oint_{|z-\zeta_j| = r} \frac{1}{z-w} \frac{\sqrt{z^2 - 4c}}{Q(z)}\, dz
  + \oint_{C_\epsilon} \frac{1}{z-w} \frac{\sqrt{z^2 - 4c}}{Q(z)}\, dz\\
  &= 2\pi i \frac{\sqrt{w^2 - 4c}}{Q(w)} + 2\pi i \sum_{j=1}^l \frac{q_j(w)}{(\zeta_j - w)^{m_j}}
    - 2i S_\mu(w). 
\end{align*}
Here $\zeta_0 = w$ satisfies $Q(w) \not= 0$, 
$\zeta_j$ ($j=1,2,\dots,l$) are zeros of $Q(z)$ in $\C \setminus I_c$ with multiplicity $m_j$, and
\[
  \left.\frac{d^{m_j-1}}{dz^{m_j-1}} \left( \frac{\sqrt{z^2-4c}}{z-w} \frac{(z-\zeta_j)^{m_j}}{Q(z)} \right)\right|_{z=\zeta_j}
  = (m_j-1)! \frac{q_j(w)}{(\zeta_j - w)^{m_j}}
\]
with $q_j(w)$ a polynomial of order $m_j - 1$. 

Note that, if $m_j = 1$, $q_j = \sqrt{\zeta_j^2-4c}/Q'(\zeta_j)$.

\begin{Remark}
If $\zeta_j \in \R$, $\zeta_j^2 \geq 4c$ because of $Q(t) > 0$ ($-2\sqrt{c} < t < 2\sqrt{c}$). 
\end{Remark}

As to the left-hand side, it is independent of $R$ whenever $R$ is large enough and defines a residue polynomial $R_Q$ of $w$ by
\[
  R_Q(w) =  \frac{1}{2\pi i} \oint_{|z| = R} \frac{1}{z-w} \frac{\sqrt{z^2 - 4c}}{Q(z)}\, dz. 
\]

Notice that 
\begin{align*}
  \frac{\sqrt{z^2-4c}}{z-w} &= z\left( 1 - \frac{2c}{z^2} - \frac{2c^2}{z^4} - \cdots\right)
                              \frac{1}{z} \left(1 + \frac{w}{z} + \frac{w^2}{z^2} + \cdots\right)\\
  &= 1 + \frac{w}{z} + \frac{w^2-2c}{z^2} + \cdots, 
\end{align*}
which is multiplied by $1/Q(z) = P(z) + p/z + O(1/z^2)$ to extract the residue $R_Q$ at $z=\infty$, i.e., the coefficient of $1/z$.

\begin{Example}
  If $Q$ is a polynomial of degree $d \geq 2$, $1/Q = O(1/z^2)$ and hence $R_Q = 0$.
  If $Q(z) = a + bz$, $1/Q = 1/bz + O(1/z^2)$ ($b \not= 0$) and hence $R_Q = 1/b$, whereas $R_Q = w/a$ ($b=0$, $a \not= 0$). 
  Thus for a polynomial $Q$, 
\[
  R_Q(w) = 
  \begin{cases}
      0 &(\deg Q \geq 2)\\
    1/b &(Q(z) = a + bz, b \not= 0)\\
    w/a &(Q(z) = a \not= 0)
  \end{cases}. 
\]
\end{Example}


\begin{Proposition}\label{residue}
  \[
    S_\mu(w) = \pi \frac{\sqrt{w^2-4c}}{Q(w)} + \pi \sum_{j=1}^l \frac{q_j(w)}{(\zeta_j - w)^{m_j}}
    - \pi R_Q(w). 
  \]  
\end{Proposition}

\begin{Example}~
  \begin{enumerate}
    \item
  For a semicircular measure $\mu$ ($Q = 2\pi c$), we have 
  $S_\mu(w) = \rho(w)/c$.
\item
  For an arcsine measure $\mu$ ($Q(t) = 4c - t^2$),
  $S_\mu(w) = -\pi/\sqrt{w^2-4c}$. 
\end{enumerate}
\end{Example}

\bigskip
\noindent
$\text{\bf deg}\, \bm{Q} \bm{=} \bm{1}$

Now let $Q(w) = N(w-a)$ ($N \not= 0$, $a \in \R$). The condition $N(t-a)>0$ ($-2\sqrt{c} < t < 2\sqrt{c}$) is then equivalent to
$Na < 0$ and $a^2 \geq 4c$, under which we see that 
\[
S_\mu(w) = \frac{1}{N} \frac{a - \sqrt{a^2 - 4c} - w + \sqrt{w^2 - 4c}}{w-a} 
\]
and the normalization $\mu(\R) = 1$ is equivalent to 
\[
N = -a + \sqrt{a^2 - 4c}.
\]
Note that $\pm N > 0$ according to $\pm a \geq 2\sqrt{c}$, whence $Na < 0$ is satisfied automatically.

We compare the above $S_\mu$ with
\[
  S_{\alpha,1/2}(w) = \frac{F_0(w) + (1/2)\sqrt{w^2 - 4c}}{G_0(w)}
  = \frac{1}{-2\alpha} \frac{2\alpha - w + \sqrt{w^2 - 4c}}{w - \alpha - c/\alpha}
\]
in Proposition~\ref{one-shift}(iii), which suggests to put $a = \alpha + c/\alpha$.
We then have 
\[
  a - \sqrt{a^2 - 4c} =
  \begin{cases}
    2\alpha &(0<|\alpha| \leq \sqrt{c})\\
    2c/\alpha &(|\alpha| \geq \sqrt{c})
  \end{cases} 
\]
and hence 
\[
  S_\mu =
  \begin{cases}
    S_{\alpha,1/2} &(0 < |\alpha| \leq \sqrt{c})\\
    S_{c/\alpha,1/2} &(|\alpha| \geq \sqrt{c}) 
  \end{cases} 
\]
with $S_{c/\alpha,1/2}$ for $0 < |\alpha| < \sqrt{c}$ equal to the Stieltjes transform of
\[
    \left(1-\frac{\alpha^2}{c}\right)\delta_{\alpha + c/\alpha} + \frac{\alpha^2}{c} \mu
\]
and $S_{\alpha,1/2}$ for $|\alpha| > \sqrt{c}$ equal to the Stieltjes transform of
\[
  \left(1-\frac{c}{\alpha^2}\right)\delta_{\alpha + c/\alpha} + \frac{c}{\alpha^2} \mu. 
\]
Here $\delta_{\alpha + c/\alpha}$ denotes the Dirac measure at $\alpha + c/\alpha \in \R \setminus I_c^\circ$ and
$\mu$ is the probability measure for the choice $a = \alpha + c/\alpha$. 

\bigskip
\bigskip
\noindent
$\text{\bf deg}\, \bm{Q} \bm{=} \bm{2}$

Let $Q(w) = N(w-a)(w-b)$ ($a,b \in \C$) be a quadratic polynomial satisfying 
$Q(t) > 0$ for $t \in I_c^\circ = (-2\sqrt{c},2\sqrt{c})$. 
Then the accompanied measure $\mu(dt) = \sqrt{4c - t^2}/Q(t)$ (supported by $I_c^\circ$) is finite unless $a=b=\pm 2\sqrt{c}$
with its Stieltjes transform $S_\mu$ calculated by Proposition~\ref{residue} to be 
\[
  \frac{N}{\pi} S_\mu(w) = \frac{\tau_\mu - \sigma_\mu w + \sqrt{w^2-4c}}{(w-a)(w-b)}. 
\]
Here 
\[
    \sigma_\mu = \frac{\sqrt{b^2-4c} - \sqrt{a^2-4c}}{b-a},
    \quad
    \tau_\mu = \frac{a\sqrt{b^2-4c} - b\sqrt{a^2-4c}}{b-a}  
\]
and the normalization condition is satisfied by $N = \pi(\sigma_\mu - 1)$.

Recall that $\sqrt{w^2-4c}$ is a holomorphic function of $w \in \C \setminus I_c$ satisfying
$\overline{\sqrt{w^2 - 4c}} = \sqrt{{\overline w}^2 - 4c}$ and, for $t \in \R \setminus I_c$, 
$\pm \sqrt{t^2-4c} > 0$ according to $\pm t > 2\sqrt{c}$.

Thus, if $a \leq -2\sqrt{c}$ and $b \geq 2\sqrt{c}$, then
\[
    \sigma_\mu = \frac{\sqrt{b^2-4c} + |\sqrt{a^2-4c}|}{b-a},
    \quad
    \tau_\mu = \frac{a\sqrt{b^2-4c} + b|\sqrt{a^2-4c}|}{b-a},   
  \]
  whereas, if $a = \zeta = \overline{b} \in \C \setminus \R$, then 
\[
    \sigma_\mu = \frac{\text{Im}\, \sqrt{\zeta^2 - 4c}}{\text{Im}\,\zeta},
    \quad
    \tau_\mu = \frac{\text{Im}\, (\overline{\zeta}\sqrt{\zeta^2 - 4c})
    }{\text{Im}\, \zeta}.  
\]

Now we relate $S_{\alpha,\beta}$ to $S_\mu$. In view of the expression
\[
  S_{\alpha,\beta} = \frac{F_0/\beta + \sqrt{w^2 - 4c}}{G_0/\beta}
  = \frac{\alpha - (1-\beta)w + \beta\sqrt{w^2 - 4c}}{G_0},
\]
the coincidence $S_\mu = S_{\alpha,\beta}$ for some $\beta \not= 1/2$ requires $\sigma = - \sigma_\mu$, which in turn implies 
$\alpha/\beta = \tau_\mu$ and hence $S_{\alpha,\beta} = S_\mu$.

\begin{Proposition}
  Given a real quadratic polynomial $(w-a)(w-b)$ ($a+b, ab \in \R$),
  there exists exactly one positive solution $(\alpha,\beta)$ which satisfies $S_{\alpha,\beta} = S_\mu$.

  For imaginary roots, there is no other positive solution.
  For real roots $a$, $b$, there is one more positive solution if $ab>0$,
  whereas there are three positive solutions other than $S_\mu$ if $ab< 0$.
  Thus there appears atomic measures for these extra solutions. 
\end{Proposition}

\bigskip
\noindent
$\text{\bf deg}\, \bm{Q} \bm{\geq} \bm{2}$

From here on assume that $Q$ is a polynomial of degree two or more.
Then
$S_\mu$ takes the form $(P_0(w) + \sqrt{w^2-4c})/Q$ with $P_0$ a real polynomial of degree $\deg Q - 1$.
If we change $P_0$ to another real polynomial $P$ of degree $\deg Q - 1$
so that it still gives the Stieltjes transform of a possibly signed measure, then
\[
  \frac{P(w) + \sqrt{w^2 - 4c}}{Q(w)} = S_\mu(w) + \frac{P(w) - P_0(w)}{Q(w)}
\]
with $(P-P_0)/Q$ having no poles inside $\text{Im}\, w > 0$. Thus, by factorizing $Q = Q_0Q_1$ with $Q_0$ covering all non-real roots of $Q$,
$P-P_0$ must be divided by $Q_0$ and we have 
\[
  \frac{P-P_0}{Q} = \frac{P_1}{Q_1} = \sum_{j=1}^d \frac{f_j(w)}{(q_j - w)^{m_j}}
\]
Here $q_j$ with multiplicity $m_j$ denotes all real roots of $Q$ and $f_j(w)$ is a polynomial of degree $m_j - 1$.
Since the existence of higher singularities violates the validity of the Stieltjes inversion formula, 
$(P-P_0)/Q$ is the Stieltjes transform of a signed measure exactly when $p_j(w)/(q_j-w)^{m_j} = p_j/(q_j-w)$ ($1 \leq j \leq d$),
\[
  \frac{P-P_0}{Q} = \sum_j \frac{p_j}{q_j - w} = \frac{p(w)}{\prod_{j=1}^d (q_j-w)} 
\]
with $p$ a real polynomial of degree $\leq d-1$ and
there remain $d$ parameters $p_j$ to choose as weights of atomic measures placed at $q_j$ ($1 \leq j \leq d$).
%
%
%
%
%
%
%
%
%
%
%
%
%
\section{Transformations of Probablity measures}

Let $\mathcal{H}$ be a complex Hilbert space with an inner product 
$\langle \, \cdot \, , \, \cdot \, \rangle$.
For vectors $u, w, f \in \mathcal{H}$ we define the rank-one operator
$u \otimes w$ by 
$$
 (u \otimes w) f = \langle \, f \, , \, w \, \rangle u.
$$

Let $A$ be a closed, densely defined operator on $\mathcal{H}$, 
$s, t \in \mathbb {C}$ and consider the sum of two rank-one 
perturbations of the form
$$
  A_{s, t} = A - s (u \otimes w) - t (g \otimes h) 
$$
for some non-zero vectors $u, w, g, h \in \mathcal{H}$.

\medskip 

In the paper \cite{KWW}, they investigated such a kind of special 
rank-two perturbations for the case of $\{g, h \} = \{u, w \}$: 
the `anti-diagonal' and the `diagonal' which are given respectively by 
$$
  \widetilde{ A_{s, t} } = A - s (u \otimes w) - t (w \otimes u) 
  \, \mbox{ and } \, 
  \widehat{ A_{s, t} }   = A - s (u \otimes u) - t (w \otimes w).
$$

\medskip 

They have shown that the `anti-diagonal' rank-two perturbation gives 
an operator model of the transformations of probability measures 
on $\mathbb{R}$ such as the $\mathrm{t}$-transform introduced 
by Bo{\. z}ejko and Wysocza{\' n}ski in \cite{BW} and the generalized 
$\mathrm{t}$-transform by Krystek and Yoshida in \cite{KY}.

\medskip 

These transformations of probability measures are based on the
changings of the Jacobi parameters, that is, the first one or two 
terms of the Jacobi parameters are modified.  Thus the rank-two 
perturbation does work well. 

\medskip 

Another such a kind of transformations of probability measures are 
investigated in some literatures. For instance, we can find 
the $V_\alpha$-transformation in \cite{KWa} and the 
$\mathrm{u}$-transform in \cite{KWW}, the operator model of which 
is given by the `diagonal' rank-two perturbation.

\medskip

We shall treat a probability measure on $\mathbb{R}$ in the frame work of 
non-commutative probability.

Let $A = A^*$ be a closed, densely defined self-adjoint operator on 
a Hilbert space $\mathcal{H}$, $u \in \mbox{Dom} (A) \subset \mathcal{H}$ 
be a unit vector $|| u || = 1$, and 
%
%
consider the vector state $\varphi_u$ on $\mathcal{B} (\mathcal{H})$ 
defined by $\varphi_u (B) =\big\langle \, B u \, , \, u \, \big\rangle$ 
for $B \in \mathcal{B} (\mathcal{H})$. 

Then the spectral theorem implies that there exists a unique 
probability measure $\mu_A$ on $\mathbb{R}$ such that 
$$
  \varphi_u \big( (z - A)^{-1} \big) =
  \int_{-\infty}^{\infty} \frac{d \mu_A (x)}{z - x}  \quad 
  z \in \mathbb{C}^{+},
$$
which is called {\it the distribution of} $A$ with respect to the 
vector state $\varphi_u$. 
We should note that $\varphi_u \big( (z - A)^{-1} \big)$ is equal to 
$G_{\mu_A} (z)$ the Cauchy transform of the distribution $\mu_A$, 
which is essentially the same as the Stieltjes transform but having 
an opposite sign.

\subsection{Finite Rank Perturbations of The Jacobi Operators}
Here we will apply such a kind of finite rank perturbations to the Jacobi 
matrix models. The Jacobi tridiagonal matrices are closely related to 
the continued fraction of the Stieltjes transform of probability measure 
and to the associated orthogonal polynomials. 

\medskip 

We shall consider the probability measure with compact support on 
$\mathbb{R}$, which ensures finite moments of all orders and covers the 
distribution of a bounded self-adjoint operator.

Let $\mu$ be a such probability measure on $\mathbb{R}$. 
Then there exists the sequence of polynomials $\big\{ P_n \big\}_{n \ge 0}$, 
which is orthonormal with respect to the measure $\mu$, that is, 
$$
    \langle P_m  \, , \,  P_n \, \rangle 
  = \int_{\mathbb{R}} P_m(x) P_n(x) \, d \mu (x) = \delta_{m,n}
$$
in the Hilbert space $L^2 \big( \mathbb{R}, d \mu (x) \big)$. 

\medskip 

The orthogonal polynomials satisfy the following three terms recurrence 
relations: 
if the Stieltjes transform $S_\mu (w)$ of the probability measure $\mu$ is 
expanded into the continued fraction of the form 
$$
 S_\mu (w) = \dfrac{     1}{a_0 - w + 
             \dfrac{-b_0^2}{a_1 - w +
             \dfrac{-b_1^2}{a_2 - w + 
             \dfrac{-b_2^2}{ \; \; \ddots \; \;
             }}}},
$$
then the polynomials satisfy the three-terms recurrence relations
$$
  x P_n (x) = b_n P_{n+1} (x) + a_n P_n (x) + b_{n-1} P_{n-1} (x) 
  \;  \mbox{ for } \;  n \ge 0,
$$
where $P_0 (x) =1$ and $b_{-1} = 0$, conventionally.
Here it should be noted that the coefficients $b_n \; (n \ge 0)$ are 
positive and $a_n  \; (n \ge 0)$ are real and bounded.

\medskip 

The Jacobi operator $J$ acting on the Hilbert space 
$L^2 \big( \mathbb{R}, d \mu (x) \big)$ is an operator of 
multiplication by the variable $x$. The matrix representation of the 
operator $J$ with respect to the orthonormal basis 
$\{ e_n := P_n \}_{n \ge 0}$ is given by the tridiagonal matrix 
$$
 J = \left[ \;
   \begin{matrix}
       a_0    & b_0    & 0      & 0      & 0      & \cdots \\
       b_0    & a_1    & b_1    & 0      & 0      & \cdots \\
       0      & b_1    & a_2    & b_2    & 0      & \cdots \\
       0      & 0      & b_2    & a_3    & b_3    & \cdots \\
       0      & 0      & 0      & b_3    & a_4    & \cdots \\
       \vdots & \vdots & \vdots & \vdots & \vdots & \ddots 
   \end{matrix}
   \right] ,
$$
which is called {\it the Jacobi matrix}.  
The Stieltjes transform $S_\mu (w)$ of the probability measure $\mu$ is,
of course, given by 
$$
  S_\mu (w) = \big\langle (J - w)^{-1} e_0 \, , \,  e_0 \, \big\rangle
              = \int_{-\infty}^{\infty} \frac{d \mu (x)}{x - w} .
$$

Here we shall introduce the special finite rank perturbations, which are 
the combination of the `anti-diagonal' perturbation in \cite{KWW} with a 
diagonal shift. These perturbations have the following form:
$$
 A  \longmapsto  A - p \, (u \otimes u) 
         - q \, (A u \otimes u) - q  (u \otimes A u)
$$ 
for some $p, q \in \mathbb{R}$. 
Applying this type of perturbation to the Jacobi operator $J$ for 
the vectors $u = e_0$ and $u = e_1$, our finite rank perturbations
are defined.

\noindent
\begin{Definition}
The finite rank perturbations $\Phi_{p,q}^{(0)}$ and $\Phi_{p,q}^{(1)}$ of the 
Jacobi operator $J$ are defined by 
$$
 J  \mapsto  \Phi_{p, q}^{(0)} \big( J \big) 
     = J - p \, (e_0 \otimes e_0) 
         - q \, (J e_0 \otimes   e_0) - q \, (  e_0 \otimes J e_0)
$$ 
and 
$$
 J  \mapsto  \Phi_{p, q}^{(1)} \big( J \big) 
     = J - p \, (e_1 \otimes e_1) 
         - q \, (J e_1 \otimes   e_1) - q \, (  e_1 \otimes J e_1),
$$ 
respectively.
\end{Definition}

\begin{Remark}
Since the perturbed operator $\Phi_{p, q}^{(0)} \big( J \big)$ can be 
expanded as 
$$
 \begin{aligned}
   \Phi_{p, q}^{(0)} \big( J \big) 
    & = J - p \, (e_0 \otimes e_0) 
          - q \, \big( (b_0 e_1 + a_0 e_0)\otimes e_0 \big) 
          - q \, \big( e_0 \otimes (b_0 e_1 + a_0 e_0) \big) \\
    & = J - p \, (e_0 \otimes e_0) 
          - 2 q a_0 \, (e_0 \otimes e_0) 
            - q b_0 \, (e_0 \otimes e_1) - q b_0 \, (e_1 \otimes e_0) \\
    & = J - (2 q a_0 + p) \, (e_0 \otimes e_0) 
            - q b_0 \, (e_0 \otimes e_1) - q b_0 \, (e_1 \otimes e_0) 
 \end{aligned} 
$$
which implies that the Jacobi matrix representation of the perturbed operator 
$\Phi_{p, q}^{(0)} \big( J \big)$ is given by 
$$
\Phi_{p, q}^{(0)} \big( J \big) 
  = \left[ \;
   \begin{matrix}
    (1 - 2 q) a_0 - p & (1 - q) b_0    & 0      & 0      & 0      & \cdots \\
    (1 - q) b_0    & a_1    & b_1    & 0      & 0      & \cdots \\
       0      & b_1    & a_2    & b_2    & 0      & \cdots \\
       0      & 0      & b_2    & a_3    & b_3    & \cdots \\
       0      & 0      & 0      & b_3    & a_4    & \cdots \\
       \vdots & \vdots & \vdots & \vdots & \vdots & \ddots 
   \end{matrix}
   \right].
$$

\medskip 

On the other hand, the perturbed operator $\Phi_{p, q}^{(1)} \big( J \big)$ is 
expanded as 
$$
 \begin{aligned}
   \Phi_{p, q}^{(1)} \big( J \big) 
    & = J - p \, (e_1 \otimes e_1) 
          - q \, \big( (b_1 e_2 + a_1 e_1 + b_0 e_0)\otimes e_1 \big) \\
    & \qquad \qquad \qquad 
          - q \, \big( e_1 \otimes (b_1 e_2 + a_1 e_1 + b_0 e_0) \big) \\
    & = J - (p + 2 q a_1) \, (e_1 \otimes e_1) 
          - q b_1  \, (e_2 \otimes e_1) - q b_1  \, (e_1 \otimes e_2) \\
    & \qquad \qquad \qquad 
          - q b_0  \, (e_0 \otimes e_1) - q b_0  \, (e_1 \otimes e_0).
 \end{aligned} 
$$
Thus the Jacobi matrix representation of the perturbed operator 
$\Phi_{p, q}^{(1)} \big( J \big)$ is given by 
$$
\Phi_{p, q}^{(1)} \big( J \big) 
  = \left[ \;
   \begin{matrix}
       a_0    & (1 - q) b_0    & 0      & 0      & 0      & \cdots \\
    (1 - q) b_0    & (1 - 2 q) a_1 - p    & (1 - q) b_1    & 0      & 0      & \cdots \\
       0      & (1 - q) b_1    & a_2    & b_2    & 0      & \cdots \\
       0      & 0      & b_2    & a_3    & b_3    & \cdots \\
       0      & 0      & 0      & b_3    & a_4    & \cdots \\
       \vdots & \vdots & \vdots & \vdots & \vdots & \ddots 
   \end{matrix}
   \right].
$$
\end{Remark}

\medskip 

\subsection{Shifts of Semicircular Root by Finite Rank Perturbations}

Let $d \sigma (x) = \dfrac{1}{2 \pi c} \sqrt{4 c - x^2} \, dx$ be the 
semicircular measure of variance $c$, and 
consider the Hilbert space $L^2 \big( \mathbb{R}, d \sigma (x) \big)$. 
Since the Stieltjes transform $S_{c} (w)$ of the semicircular measure of 
variance $c$ is expanded into the continued fraction of the form
$$
   S_{c} (w) = 
            \dfrac{ 1}{ -w + 
            \dfrac{-c}{ -w + 
            \dfrac{-c}{ -w + 
            \dfrac{-c}{ \; \; \ddots \; \;
            }}}},
$$
the corresponding the Jacobi matrix is given by 
$$
  J_{c}
  = \left[ \;
   \begin{matrix}
       0      & \sqrt{c} & 0        & 0        & 0        & \cdots \\
     \sqrt{c} & 0        & \sqrt{c} & 0        & 0        & \cdots \\
       0      & \sqrt{c} & 0        & \sqrt{c} & 0        & \cdots \\
       0      & 0        & \sqrt{c} & 0        & \sqrt{c} & \cdots \\
       0      & 0        & 0        & \sqrt{c} & 0        & \cdots \\
       \vdots & \vdots & \vdots & \vdots & \vdots & \ddots 
   \end{matrix}
   \right].
$$

As it is mentioned in {Section 3} that the one-step shift 
$S_{\alpha, \beta} (w)$ and the two-step shift $S_{2} (w)$ of the 
semicircular measure are given by the liner fractional transforms
$$
   S_{\alpha, \beta} (w) = \dfrac{1}{\alpha - w + 2 \beta S_c (w)}
$$
and 
$$
   S_{2} (w) = \dfrac{1}{\gamma - w + \delta S_{\alpha, \beta}(w)},
$$
respectively. Here we will see that the one-step and the two-step
{shifts} of the semicircular measure can be obtained by applying 
our finite rank perturbations to the corresponding Jacobi operators.

\begin{Proposition}
The Jacobi matrix $J_{\alpha, \beta}$ corresponding to the one-step 
shift $S_{\alpha, \beta}$ is given by the rank-two perturbation 
of $J_c$ that 
$$
   J_{\alpha, \beta} = \Phi_{p,q}^{(0)} (J_c),
$$
where $(p, q) = \big( -\alpha, \,  1 - \sqrt{-2 \beta} \, \big)$.
\end{Proposition}

\begin{proof}
It is easy to find that the perturbed Jacobi matrix $\Phi_{p,q}^{(0)} (J_c)$ 
becomes
$$
\Phi_{p,q}^{(0)} (J_c)
  = \left[ \;
   \begin{matrix}
    - p  &  (1 - q) \sqrt{c} & 0        & 0        & 0        & \cdots \\
 (1 - q) \sqrt{c} & 0        & \sqrt{c} & 0        & 0        & \cdots \\
       0      & \sqrt{c} & 0        & \sqrt{c} & 0        & \cdots \\
       0      & 0        & \sqrt{c} & 0        & \sqrt{c} & \cdots \\
       0      & 0        & 0        & \sqrt{c} & 0        & \cdots \\
       \vdots & \vdots & \vdots & \vdots & \vdots & \ddots 
   \end{matrix}
   \right].
$$
Moreover $S_{\alpha, \beta} (w)$ has the continued fraction of the form 
$$
   S_{\alpha, \beta} (w) = 
            \dfrac{ 1}{ \alpha - w + 
            \dfrac{ 2 \beta c}{ -w + 
            \dfrac{-c}{ -w + 
            \dfrac{-c}{ \; \; \ddots \; \;
            }}}}.
$$
Hence we have $\alpha = -p$ and $\sqrt{-2 \beta} = 1 - q$.
\end{proof}

\medskip 

\begin{Proposition}
The Jacobi matrix $J_{2}$ corresponding to the two-step shift 
$S_{2}$ of the semicircular is given by the composition of 
the finite rank perturbations of $J_c$ that 
$$
   J_{2} = \Phi_{p_1,q_1}^{(1)} \big( \Phi_{p_0,q_0}^{(0)} (J_c) \big),
$$
where 
  $(p_0, q_0) = \Big( -\gamma, \,  1 - \sqrt{\sfrac{\delta}{2 \beta c}} \Big)$ 
and 
  $(p_1, q_1) = \big( -\alpha, \,  1 - \sqrt{-2 \beta} \big)$.
\end{Proposition}

\begin{proof}
It can be found that $S_{2} (w)$ has the continued fraction of the form 
$$
   S_{2} (w) = 
            \dfrac{ 1}{ \gamma - w + 
            \dfrac{ \delta }{\alpha  -w + 
            \dfrac{2 \beta c}{ -w + 
            \dfrac{-c}{ \; \; \ddots \; \;
            }}}}.
$$
On the other hand, the perturbed Jacobi matrix 
$\Phi_{p_1,q_1}^{(1)} \big( \Phi_{p_0,q_0}^{(0)} (J_c) \big)$ becomes 
$$
\begin{aligned}
& \Phi_{p_1,q_1}^{(1)} \big( \Phi_{p_0,q_0}^{(0)} (J_c) \big) = \\
& \left[ \;
   \begin{matrix}
      - p_0  &  (1 - q_0)(1 - q_1) \sqrt{c} & 0        & 0        & 0        & \cdots \\
 (1 - q_0)(1 - q_1) \sqrt{c} & - p_1      & (1 - q_1)\sqrt{c} & 0        & 0        & \cdots \\
       0      & (1 - q_1)\sqrt{c} & 0        & \sqrt{c} & 0        & \cdots \\
       0      & 0        & \sqrt{c} & 0        & \sqrt{c} & \cdots \\
       0      & 0        & 0        & \sqrt{c} & 0        & \cdots \\
       \vdots & \vdots & \vdots & \vdots & \vdots & \ddots 
   \end{matrix}
   \right].
\end{aligned}
$$
Hence it follows that 
$\gamma = -p_0$, $\sqrt{-\delta} = (1 - q_0)(1 - q_1) \sqrt{c}$,
              $\alpha = -p_1$, and $\sqrt{-2 \beta} = 1 - q_1$, 
which implies 
$$
   p_0 = - \gamma, \; 
   q_0 = 1 - \sqrt{\sfrac{\delta}{2 \beta c}}, \; 
   p_1 = \alpha, \; 
   q_1 = 1 - \sqrt{-2 \beta}.
$$
\end{proof}

\begin{Remark}
If we define the $k$th level perturbation $\Phi_{p,q}^{(k)}$ of the Jacobi 
operator $J$ by 
$$
 J  \mapsto  \Phi_{p, q}^{(k)} \big( J \big) 
     = J - p \, (e_k \otimes e_k) 
         - q \, (J e_k \otimes e_k) - q \, (  e_k \otimes J e_k) \quad 
  (k \ge 0)
$$ 
then the $n$-shift of semicircular measure can be obtained by applying 
the finite rank perturbation which is given by the 
composition of $\Phi_{p,q}^{(k)}$ ($k=0,1, \ldots, n-1$), that is 
$$
      \Phi_{p_{n-1},q_{n-1}}^{(n-1)} \big( 
      \Phi_{p_{n-2},q_{n-2}}^{(n-2)} \big( \cdots 
      \Phi_{p_{1},q_{1}}^{(1)} \big(
      \Phi_{p_{0},q_{0}}^{(0)} \big(J_c \big)\big) \cdots \big)\big)
$$
for certain parameters $\big\{ ( p_k, q_k ) \big\}_{k = 0}^{n-1}$.
\end{Remark}

\bigskip

\end{document}